\documentclass[11pt%, draft
]{amsart}
\usepackage{amsmath, amssymb} 

\usepackage{hyperref}
\usepackage{mathrsfs}

\usepackage{esint}

\usepackage{cancel}
	\normalbaselines						  %

\numberwithin{equation}{section}

\newtheorem{thm}{Theorem}[section]

\newtheorem{lm}[thm]{Lemma}

\theoremstyle{remark}
\newtheorem{rem}[thm]{Remark}

\newtheorem*{rem*}{Remark}

\theoremstyle{definition}

\newtheorem*{df*}{Definition}

\newenvironment{entry}
{\begin{list}{X}%
  {%
      \setlength{\labelwidth}{55pt}%
      \setlength{\leftmargin}{\labelwidth}%\labelsep}%
      \addtolength{\leftmargin}{\labelsep}%
      \setlength{\itemsep}{.4pc}
   }%
}%
{\end{list}}

\newcounter{vremennyj}

\newcommand\cond[1]{\setcounter{vremennyj}{\theenumi}\setcounter{enumi}{#1}\labelenumi\setcounter{enumi}{\thevremennyj}}

\newcommand{\cz}{Cald\-er\-\'{o}n--Zygmund\ }

\newcommand{\la}{\lambda}

\newcommand{\T}{\mathbb{T}}

\newcommand{\f}{\varphi}

\newcommand{\e}{\varepsilon}

\newcommand{\C}{\mathbb{C}}
\newcommand{\R}{\mathbb{R}}
\newcommand{\Z}{\mathbb{Z}}
\newcommand{\N}{\mathbb{N}}
\newcommand{\E}{\mathbb{E}}

\newcommand{\1}{\mathbf{1}}

\newcommand{\La}{\langle }
\newcommand{\Ra}{\rangle }

\newcommand{\lnm}{\!\left\bracevert\!}
\newcommand{\rnm}{\!\right\bracevert\!\!}

\newcommand{\Om}{\Omega}
\newcommand{\om}{\omega}

\newcommand{\cH}{\mathcal{H}}

\newcommand{\cL}{\mathcal{L}}
\newcommand{\cK}{\mathcal{K}}

\newcommand{\cD}{\mathcal{D}}
\newcommand{\cX}{\mathcal{X}}
\newcommand{\cY}{\mathcal{Y}}
\newcommand{\bP}{\mathbb{P}}

\newcommand{\bk}{\mathbf{k}}
\newcommand{\bS}{\mathbf{S}}
\newcommand{\fb}{\mathbf{\dot f}}

\newcommand{\fdot}{\,\cdot\,}

\newcommand{\ci}[1]{_{ {}_{\scriptstyle #1}}}
\newcommand{\ti}[1]{_{\scriptstyle \text{\rm #1}}}
\newcommand{\ut}[1]{^{\scriptstyle \text{\rm #1}}}
\newcommand{\wt}{\widetilde}

\begin{document}

\title%[]
{$H^1$ and dyadic $H^1$}

\author{Sergei Treil}
\dedicatory{To Victor Petrovich Khavin}
\thanks{This research is  partially supported by the NSF grant DMS-0800876.
}
\address{Department of Mathematics, Brown University, 151 Thayer 
Str./Box 1917,      
 Providence, RI  02912, USA }
\email{treil@math.brown.edu}
\urladdr{http://www.math.brown.edu/\~{}treil}

\date{}

\begin{abstract} 
In this paper we give a simple proof of the fact that the average over all dyadic lattices of the dyadic $H^1$-norm of a function gives an equivalent $H^1$-norm. The proof we present works for both one-parameter and multi-parameter Hardy spaces. 

The results of such type are known; cf.~\cite{Davis-Rearrangements-1980} for the one-parameter case. Also, by duality, such results are equivalent to the ``BMO from dyadic BMO'' statements proved in 
\cite{Garn-Jones-BMO-dyadic-1982} (one parameter case) and \cite{Pipher-Ward-2008} (two-parameter case). 

While the paper generalizes these results to the multi-parameter setting, this is not its main goal. The purpose of the paper is to present an approach leading to a simple proof, which works in both one-parameter and multi-parameter cases. 

The main idea of treating square function as a \cz operator is  a commonplace in harmonic analysis; the main observation, on which  the paper is based, is that one can treat the random dyadic square function this way. 
After that, all is proved by using the standard and well-known results about \cz operators in the Hilbert-space-valued setting. 

As an added bonus, we get a simple proof of the (equivalent by duality) inclusions $\text{BMO}\subset \text{BMO}\ti d$, $H^1\ti d \subset H^1$ in the multi-parameter case. Note, that unlike the one-parameter case, the inclusions in the general situation are far from trivial. 

\end{abstract}

\maketitle

\setcounter{tocdepth}{1}
\tableofcontents

\setcounter{section}{-1}

\section*{Notation}

\begin{entry}

\item[$\cH\otimes \cK$] tensor product of Hilbert spaces $\cH$ and $\cK$; we consider its endowed with the ``canonical'' norm, which makes $\cH\otimes \cK$ a Hilbert space.  If one of the spaces is a function space, for example if $\cH=L^2$, then $\cH\otimes \cK$ can be interpreted as $L^2$ with values in $\cK$. 

\item[$L^p(X; \cH)$] $L^p$-space of functions on $X$ with values in $\cH$; $X$ is usually $\R^N$ with the Lebesgue measure. Symbols $X$ and/or $\cH$ can be omitted if they are clear from the context. 

\item[$dx$] Lebesgue measure in $\R^N$ if $x\in \R^N$. 

\item[$\cD$] Dyadic Lattice in $\R^N$, see below. 

\item[$H^1$, $H^1_{\cD}$] real variable Hardy space on $\R^N$ and its dyadic counterpart, associated to a dyadic lattice $\cD$. We will also consider Hilbert-space-valued spaces, and the notation like $H^1(X; \cH)$ will be used. 

\item[$\ell(Q)$, $|Q|$] sidelength and volume of a cube $Q\subset \R^N$. 

\item[$\E\xi$]  expectation of a random variable $\xi$, $\E\xi = \int_\Omega \xi(\omega) dP(\om)$; sometimes, to distinguish the random variable in a formula, the notation $\E\xi(\om)$ or $\E_\om\xi(\om)$ will be used instead of $\E\xi$. 

\end{entry}

\subsection*{Cubes and dyadic lattices}
Throughout the paper we will speak a lot about dyadic cubes and dyadic
lattices, so let us first fix some terminology. A {\em cube} in
$\R^N$ is an object obtained from the {\em standard cube} $[0,1)^N$ by
dilations and shifts.

For a cube $Q$ we denote by $\ell(Q)$ its {\em size}, i.~e. the length
of its side. Given a cube $Q$ one can split it by dividing each side in halves into $2^N$ cubes $Q_k$
of size $\ell(Q)/2$: we will call such cubes $Q_k$ the {\em children
of $Q$}.

For a cube $Q$ and $\la >0$ we denote by $\la Q$ the cube $Q$ dilated
$\la$ times with respect to its center.

Now, let us define the {\em standard dyadic lattice} $\cD_0$: for
each $k\in\Z$ let us consider the cube $[0, 2^k)^N$ and all its shifts
by elements of $\R^N$ with coordinates of form $j\cdot 2^k$,
$j\in\Z$. The collection of all such cubes (union over all $k$) is
called the {\em standard dyadic lattice}.

A {\em dyadic lattice} $\cD$ is just a shift of the standard dyadic
lattice $\cD_0$.
A collection of all cubes from a dyadic lattice $\cD$ of a fixed size
$2^k$ is called a {\em layer} of the lattice.

\subsection*{Random dyadic lattice}

Our random lattice will contain the dyadic cubes  of standard size
$2^k$ ($k\in\Z$), but will be ``randomly shifted'' with respect to the
standard dyadic lattice $\cD_0$. The simplest idea would be to
pick up a random variable
$x$ uniformly distributed over $\R^N$ and to define the random
lattice as
$x+\cD_0$. This idea works for the torus $\T^N$, but unfortunately, there exists no such random variable $x$ in $\R^N$, so we have to
act in a little bit more sophisticated way.

Let us construct a random lattice of dyadic intervals
on the real line $\R$, and then define a random lattice in $\R^N$ as
the product of the lattices of intervals.

Let $\Omega_1$ be some probability
space and let
$x(\omega)$ be a random  variable
uniformly distributed over the interval $[0,1)$.

Let $\xi_j(\om)$ be
random variables satisfying
$\bP\{\xi_j=+1\}=\bP\{\xi_j=-1\}=1/2$. Assume also that
$x(\om),\xi_j(\om)$, $j\in\N$ are
independent. Define the random lattice $\cD(\om)$ as follows:

\begin{enumerate}
\item
We require that  $I_0(\om):=[x(\om)-1,x(\om)]\in\cD(\om)$; this  gives us
all
intervals in $\cD(\om)$ of length $2^k$, $k\le0$.

\item To determine the rest of the intervals, it is enough to know dyadic intervals $I_k(\om)\supset I_0(\om)$, of length $2^k$, $k\ge 0$. 
The intervals $I_k(\om)$  are
determined inductively: if $I_{k-1}(\om)\in\cD(\om)$ is already known (and thus all intervals of length $2^{k-1}$ in $\cD(\om)$), then 
\begin{itemize}
	\item $I_{k}(\om)$ is the union of $I_{k-1}(\om)$ and its right neighbor if 
$\xi_k(\om)=+1$,
and 
  \item $I_k(\om)$ is the union of $I_{k-1}(\om)$ and its left neighbor if $\xi_k(\om)=-1$. 
\end{itemize}
\end{enumerate}

To get a random dyadic lattice in $\R^N$ we just take  $N$
independent random dyadic lattices $\cD_1, \cD_2, \ldots, \cD_N$ in $\R$ and consider all cubes $Q=I_1\times I_2\times \ldots \times I_N$, $I_k\in \cD_k$.

\section{Introduction and main results}

\subsection{One parameter case}
\label{s0.1}
Let $H^1=H^1(\R^N)$ be the usual \emph{real variable} Hardy space on $\R^N$, and let $H^1\ci\cD$ be its dyadic counterpart, defined as follows. 

Consider a dyadic lattice $\cD$ in $\R^N$. 
Let $E_k=E_k^\cD$ be the averaging operator over cubes $Q\in \cD$ of size $2^k$, 
$
E_k f(x) = |Q|^{-1}\int_Q f
$, 
where $Q$ is the cube in $\cD$ of size $2^k$ containing $x$. If $Q\in \cD$ is a cube of size $2^k$ define $E\ci Q$ by $E\ci Q f = (|Q|^{-1}\int_Q f ) \1\ci Q = \1\ci Q (E_k f)$. 
%In other words $E\ci Q f$ is the restriction of $E_k f$ onto $Q$. 

Define the martingale differences $\Delta_k = \Delta_k^\cD := E_{k-1}^\cD - E_k^\cD$, and again for a cube $Q\in \cD$ of size $2^k$ define $\Delta\ci Q$ by $\Delta\ci Q f = \1\ci Q (E_k f) $. 

Define the dyadic square function $S=S\ci\cD$ by
\begin{equation}
\label{0.1}
(S\ci\cD f)(x) = \left(\sum_{k\in \Z} \left| \Delta_k^\cD f(x) \right|^2 \right)^{1/2} 
=\left(\sum_{Q\in \cD, Q\ni x} \left| \Delta_Q f(x) \right|^2 \right)^{1/2}. 
\end{equation}

One can also consider a slightly different square function $\wt Sf = \wt S\ci\cD f$
\begin{equation}
\label{0.2}
(\wt S\ci\cD f)(x) = \left(\sum_{k\in \Z} \left( E_k^\cD \left| \Delta_k^\cD f \right|^2  \right)(x) \right)^{1/2}
= \left(\sum_{Q\in \cD, \, Q\ni x} \left( E_Q^\cD \left| \Delta_Q^\cD f \right|^2 \right)(x)\right)^{1/2};
\end{equation}
in other words, each term in the sum on the right equals the average of $|\Delta\ci Q f|^2$ if $x\in Q$ and $0$ otherwise. 

It is not hard to show and will be explained later that $C^{-1} \| Sf \|_1 \le \| \wt Sf \|_1 \le C\| Sf \|_1 $, where the constant $C$ depends only on the dimension $N$. 

\begin{df*}
A function $f\in L^1\ti{loc}$ is in the dyadic Hardy space $H^1\ci\cD$ (with respect to the dyadic lattice $\cD$) if $\|f\|_{H^1_\cD}:= \| S\ci\cD f\|_1<\infty$
\end{df*}

One can use the square function $\wt S_\cD f$ here and get an equivalent norm. 

And now one of the main results of the paper. Let $\cD(\om)$, $\om\in \Omega$ be the random dyadic lattice, as described above, and let us recall that  $\E=\E_\om$ denotes the expectation (average with respect to $\om$) 

\begin{thm}
\label{t0.1}
A function $f\in L^1\ti{loc}(\R^n)$ belongs to $H^1$ if and only if 
$$
\int_{\R^N} \left[ \E (| S_{\cD(\omega)} f (x) |^2 )\right]^{1/2} dx <\infty.
$$
Moreover, the latter quantity gives an equivalent norm on $H^1$. 

The same result is true with the square function $\wt S_{\cD(\om)}$ in place of $S_{\cD(\om)}$. 
\end{thm}

\subsection{Multi-parameter case}
\label{s0.2}

The above results can be generalized to the case of multi-parameter Hardy spaces. Let $X_k =\R^{N_k}$, $k=1, 2, \ldots, n$ and let $H^1(X_1\times X_2\times \cdots \times X_n)$ be the $n$-parameter Hardy space, see Section \ref{s2.1} below for the definition. 

Define the dyadic Hardy space as follows. Let $\cD_k$ be a dyadic lattice on $X_k$, $k=1, 2, \ldots, n$ and let $\cD = \cD_1\times \cD_2\times\ldots\times \cD_n$ be the product dyadic lattice on $X=X_1\times X_2\times \ldots \times X_n$; the elements on $\cD$ are the ``rectangles'' (parallelepipeds) $R=Q_1\times Q_2\times\ldots \times Q_n$, $Q_k\in \cD_k$. 

For a multiindex $\bk=(k_1, k_2, \ldots k_n)$ define on $X=X_1\times X_2\times \cdots \times X_n$ the average $E_\bk:= E^1_{k_1}E^2_{k_2}\ldots E^n_{k_n}$ and the  martingale differences $\Delta_\bk := \Delta^1_{k_1} \Delta^2_{k_2} \ldots \Delta^n_{k_n}$, where $E^j_{k_j}$ and $\Delta^j_{k_j}$ are the ``one variable''   averages $E_{k_j}$ martingale difference $\Delta_{k_j}$ as defined above by 
in Section \ref{s0.1},   
taken in the variable $x_j\in X_j$. 

For a ``rectangle'' $R=Q_1\times Q_2\times\ldots \times Q_n$, $Q_k\in \cD_k$ define $\Delta_R := \Delta^1_{Q_1}\Delta^2_{Q_2} \ldots\Delta^n_{Q_n}$, where again $\Delta^j_{Q_j}$ is the operator $\Delta_{Q_j}$ defined by \eqref{0.2}
%above in Section \ref{s0.1} 
taken in the variable $x_j$. 

Define the multi-parameter square function $S=S_\cD$ by 
\begin{equation}
\label{0.3}
(S\ci\cD f)(x) = \left(\sum_{\bk\in \Z^n} \left| \Delta_\bk f(x) \right|^2 \right)^{1/2} 
=\left(\sum_{R\in \cD, R\ni x} \left| \Delta_R f(x) \right|^2 \right)^{1/2}. 
\end{equation}
One can also define the square function $\wt S =\wt S_\cD$
\begin{equation}
\wt S f (x) 	:= \left(\sum_{k\in \Z} \left( E_k \left| \Delta_k f \right|^2 \right)(x)\right)^{1/2} = 
\left(\sum_{R\in \cD, R\ni x} \left( E_R \left| \Delta_R f \right|^2 \right)(x) \right)^{1/2}
\end{equation}
The definitions look very similar to \eqref{0.1}, \eqref{0.2}, only here the sums in the right hand side is taken over all ``rectangles'' $R$, while in \eqref{0.1}, \eqref{0.2} they are taken over all cubes.

We use the same notation for the one-parameter and multi-parameter square function, but since we will treat these cases in different sections, we hope to avoid the confusion. 

\begin{df*}
Let $\cD = \cD_1\times \cD_2\times\ldots\times \cD_n$ be a product dyadic lattice on $X=X_1\times X_2\times \ldots \times X_n$. We say that a function $f\in L^1\ti{loc}(X)$ belongs to the dyadic Hardy space $H^1_\cD$ if $\|f\|_{H^1_\cD}:= \| S\ci\cD f\|_1<\infty$
\end{df*}

Let now $\cD_1(\om), \cD_2(\om), \ldots, \cD_n(\om)$ be the independent random  dyadic lattices on the spaces $X_1, X_2, \ldots, X_n$ respectively, and let $\cD(\om)= \cD_1(\om)\times \cD_2(\om)\times \ldots\times \cD_n(\om)$ be the multi-parameter random dyadic lattice.

\begin{thm}
\label{t0.4}
A function $f\in L^1\ti{loc}(\R^n)$ belongs to $H^1(X_1\times X_2\times \cdots \times X_n)$ if and only if 
$$
\int_{\R^N} \left[ \E (| S_{\cD(\omega)} f (x) |^2 ) \right]^{1/2} dx <\infty;
$$
here $S_{\cD(\om)}$ is the multi-parameter random dyadic lattice defined above. Moreover, the latter quantity gives an equivalent norm on $H^1$. 

The same result is true with the square function $\wt S_{\cD(\om)}$ in place of $S_{\cD(\om)}$. 
\end{thm}

\subsection{Some remarks}

\begin{rem}
It is a simple exercise to show that  in both one-parameter and multi-parameter cases
$$
\E (| S_{\cD(\omega)} f (x) |^2 =\lim_{Q\to \R^N} \frac1{|Q|} \int_Q |S_\cD f(x+u)|^2 du
$$  
where $\cD$ is a fixed dyadic lattice and the limits is taken over cubes (or parallelipipeds) $Q\subset \R^N$ ($Q\subset X$ in the multi-parameter case) centered at $0$ whose sidelengths tend to $\infty$. 
\end{rem}

\begin{rem}
Note that for every dyadic lattice $\cD$ one has $\|f\|_{H^1} \le C \| S_{\cD} f\|_1$, thus  by Tonelli   theorem and H\"{o}lder inequality we have for $p\le 2$
$$
\|f\|_{H^1} \le C \int_{\R^N} \left[ \E ( | S_{\cD(\omega)} f (x) |^p ) \right]^{1/p} dx 
\le C \int_{\R^N} \left[ \E (| S_{\cD(\omega)} f (x) |^2 ) \right]^{1/2} dx. 
$$
In the one-parameter case the estimate $\|f\|_{H^1} \le C \| S_{\cD} f\|_1$ is trivial and well known. Indeed, $H^1$-BMO duality ($(H^1)^*=\text{BMO}$, $(H^1_\cD)^*=\text{BMO}_\cD$, see the definitions below in Section \ref{s3.1}) and the trivial inclusion $\text{BMO}\subset\text{BMO}_\cD$ imply the $H^1_\cD \subset H^1$ with the desired estimates of the norms.  

Since $\text{BMO}\ne\text{BMO}\ci\cD$, one can conclude that th inclusion $H^1_\cD \subset H^1$ is proper. 

In the multi-parameter case the estimate $\|f\|_{H^1} \le C \| S_{\cD} f\|_1$is also known to specialists, but it is less trivial. In fact, the only place that the author is aware of, where this is proved is the Ph.D.~thesis of J.~Pipher; this proof is far from trivial and the calculations are quite tedious.  Below in Section  \ref{s2.3} we present a different, quite simple prove of this fact. This proof is based on one-parameter (Hilbert-space-valued) $H^1$-BMO theory. 
\end{rem}

\begin{rem}
Applying H\"{o}lder inequality to Theorem \ref{t0.1} (to Theorem \ref{t0.4} in the multi-parameter case) we get that for $p\le 2$
\begin{equation}
\label{0.5}
 \int_{\R^N} \left[ \E ( | S_{\cD(\omega)} f (x) |^p ) \right]^{1/p} dx 
\le  \int_{\R^N} \left[ \E (| S_{\cD(\omega)} f (x) |^2 ) \right]^{1/2} dx
\le C \|f\|_{H^1}
\end{equation}
If $p=1$ Tonelli Theorem and \eqref{0.5} implies that 
\begin{equation}
\label{0.6}
\E\|S\ci{\cD(\om)} f\|_1 \le C\|f\|_{H^1}.
\end{equation}
This was proved in one-parameter case in \cite{Davis-Rearrangements-1980}. Note, that by duality \eqref{0.6} is equivalent to ``BMO from dyadic BMO'' statements: if $f_\om$ $\om \in \Om$ is a measurable family of functions, $f_\om \in \text{BMO}_{\cD(\om)}$, $\|f_\om\|\ci{\text{BMO}_{\cD(\om)}}\le 1$, then for $f$ defined by $f(x) =\E_\om f_\om (x)$ one has $f\in \text{BMO}$, $\|f\|\ti{BMO} \le C$. 

This ``BMO from dyadic BMO'' result was proved directly in \cite{Garn-Jones-BMO-dyadic-1982}  (one-parameter case) and  in the recent paper \cite{Pipher-Ward-2008}  (the two-parameter case). 
\end{rem}

\section{Proof of Theorem 
\ref{t0.1}}
\label{s1}

\subsection{``Vectorization'' of the square function}
There is a standard way of making the nonlinear operator $S_\cD$ into a linear one by treating $S_\cD f$ as a vector-valued function. 

Namely, let us consider the space $L^1(\ell^2)$ consisting of functions on $\Z\times \R^N$ such that 
$$
\| f\|_{L^1(\ell^2)} := \int_{\R^N} \left( \sum_{k\in\Z} | f(k, x)|^2 \right)^{1/2} dx<\infty. 
$$

Define the vector-valued square function $\bS_\cD f$ by
$$
\bS_\cD f (k, x) := \Delta_k^\cD f (x), \qquad k\in \Z,\ x\in \R^N
$$
($\bS_\cD f$ is a function on $\Z\times \R^N$). We will treat this function as a function of the argument $x\in \R^N$ with values in $\ell^2=\ell^2(\Z)$. 

Clearly, $\lnm \bS_\cD f (\fdot, x)\rnm_{\ell^2} = S_\cD f(x)$, so $f\in H^1_\cD$ if and only if $\bS f \in L^1(\ell^2)$. Moreover,  $\| \bS_\cD f \|_{L^1(\ell^2)} = \|S_\cD f\|_1 = \| f\|_{H^1_\cD}$. 

Let now $\cD(\om)$ be the random dyadic lattice, and let $\Om, \bP$ be the corresponding probability space. Consider the space $\cL= L^1(\ell^2\otimes L^2(\Om, \bP))$, 
$$
\| f\|_{\cL} := \int_{\R^N} \left( \int_\Om \sum_{k\in\Z} | f(k,\om, x)|^2 d\bP(\om) \right)^{1/2} dx. 
$$
It is an $L^1$ space with values in the Hilbert space $\ell^2\otimes L^2(\Om, \bP)$; here again $f$ is treated as function of the argument $x\in \R^N$ with values in $\ell^2\otimes L^2(\Om, \bP)$. 

Define the vector-valued square function $\bS$ with values in $\ell^2\otimes L^2(\Om, \bP)$ by
\begin{equation}
\label{1.1}
\bS f(k, \om, x) = \bS_{\cD(\om)} f (k, x), \qquad x\in\R^N, \ k\in \Z, \om\in \Om. 
\end{equation}
here and below we will use notation $\bS(\fdot, \fdot, x)=: \bS f (x) \in \ell^2\otimes L^2(\Om, \bP)$.

Clearly, $\|\bS f(x)\|_{\ell^2\times L^2(\Om, \bP)} = \left[\E (| S_{\cD(\omega)} f (x) |^2 )\right]^{1/2}$, so 
$\|\bS f\|\ci{\cL} = \int_{\R^N} \left[\E (| S_{\cD(\omega)} f (x) |^2 )\right]^{1/2} dx$ (the norm in $\cL$ was constructed so this would hold). 

Similar ``vectorization'' can be performed to the square function $\wt S$. Namely, fix some ordering of the ``children'' of a dyadic cube (the same one for all cubes), and for each dyadic cube $Q$ define operator from the functions constant on the ``children'' of $Q$ to $\C^{2^N}$ by 
$$
U_Q\1_{Q_k} = |Q_k|^{1/2} e_k, \qquad k=1, 2, \ldots, 2^{2^N}, 
$$
where $Q_k$ are ``children'' of $Q$ and $\{e_k\}_{k=1}^{2^N}$ is the standard basis in $\R^{2^N}$. 
Let now $\cL:= L^1(\ell^2\otimes L^2(\Om, \bP)\otimes \C^{2^N})$ with the norm 
$$
\| f\|_{\cL} := \int_{\R^N} \left( \int_\Om \sum_{k\in\Z} \lnm f(k,\om, x)\rnm\ci{\C^{2^N}}^2 d\bP(\om) \right)^{1/2} dx. 
$$

Define the square function $\wt \bS$ by 
$$
\wt\bS f (k, \om, x) :=  U_Q \Delta_Q f \in \C^{2^N}, 
$$
where $Q\in \cD(\om)$ is the cube of size $2^k$ containing $x$. 
From the construction it is clear that $\|\wt\bS f\|\ci{\cL}  = \int_{\R^N} \left[\E (| \wt S_{\cD(\omega)} f (x) |^2 )\right]^{1/2} dx$.

Now the proof of Theorem \ref{t0.1} can be outlined in few sentences. First, it is nor hard to show that $\bS$ (or $\wt\bS$) is a \cz operator, whose kernel takes values in the Hilbert space $\ell^2\otimes L^2(\Om, \bP)$.   It is a well known fact that such \cz operators map $H^1$ to $L^1$, so
$$
\int_{\R^N} \left[ \E (| S_{\cD(\omega)} f (x) |^2 )\right]^{1/2} dx = \|\bS f\|\ci{\cL} \le C \|f\|_{H^1}. 
$$
As we discussed above in Section \ref{s0.1}, the opposite inequality $ \|f\|_{H^1} \le C \|\bS f\|\ci{\cL}$ is trivial.

Of course, the classical theory of \cz operators deals with the scalar-valued kernels. But, as it is well known to the specialists, all the facts that we need, are valid in the case of Hilbert space valued kernels too. 

However, the blind trust is not expected from the reader: all relevant facts will be presented below. 

\subsection{$\bS$ as a \cz operator}
\label{s1.1}
Let us recall that the classical \cz kernel on $\R^N$ is the function $K(\fdot, \fdot)$ defined on $\R^N\times \R^N\setminus\{(x, x):x\in\R^N\}$
satisfying  
\begin{enumerate}
\item $| K(x,y)|\le C |x-y|^{-N}$;
\item There exists $\delta>0$ such that
    $$
| K(x,y)-K(x_0,y)|,\ \lnm K(y,x)-K(y,x_0)\rnm \le
C  \frac{|x-x_0|^\delta}{|y-x_0|^{N+ \delta}}\ ,
$$
 whenever $|y-x_0|\ge2|x-x_0|$
\end{enumerate}

One can also consider operator-valued kernels, $K(x, y)\in B(X, Y)$ for arbitrary Banach spaces $X$ and $Y$. In this case $|\fdot |$ in the left hand side should be replaced by the norm in $B(\cX, \cY)$.

\subsubsection{Vector-valued \cz operators}
\label{s1.2.1}
A \cz operator with (op\-er\-a\-tor-valued) \cz kernel $K$ is a bounded operator $T:L^2(\R^N;\cX)\to \linebreak L^2(\R^N, \cY)$ such that for all compactly supported $f\in L^2(\R^N;\cX)$ and $g\in L^2(\R^N;\cY^*)$ with separated supports 
$$
\La Tf, g\Ra=\iint_{\R^N\times\R^N} \left\La K(x, y) f(y), g(y)\right\Ra dy dx . 
$$
Such operators with operator-valued kernels were considered, for example, in \cite{RdFra-OpValuedCZO-1986}, and it was proved there (see Theorem 1.2 in Ch.~3) that if $\cX$ and $\cY$ are Hilbert spaces then the operator acts from $H^1$ to $L^1$, $\| Tf\|_{L^1(\R^N, \cX)} \le C \| f \|_{H^1(\R^N, \cY)}$. 

In fact, in \cite{RdFra-OpValuedCZO-1986} a much more general situation was considered. The kernel was only assumed to satisfy a weaker version of condition \cond2, and there was no condition \cond1.  Moreover, it was assumed that $T$ was bounded in some $L^p$, $1<p<\infty$, not necessarily  $p=2$. 

Note, that condition \cond1 in some form is required for the proof of $T1$ and $Tb$ theorems, but if one assumes that $T$ is bounded in some $L^p$, $1<p<\infty$ the condition \cond2 alone is sufficient for the action from $H^1$ to $L^1$.

Also, in \cite{RdFra-OpValuedCZO-1986} the theorem was proved for the case when $\cX$ and $\cY$ are arbitrary Banach spaces, if one defines $H^1$ via atomic decomposition.

% the assumptions on about the kernel were weaker, it was assumed that $T$ was bounded in some $L^p$, $1<p<\infty$, not necessarily for $p=2$.  

It is well known that for the case of Hilbert-space-valued functions all the definitions of $H^1$ (via atomic decompositions, via maximal function, via different square functions, via Riesz transforms) 
are equivalent%
\footnote{Unfortunately, the author cannot point to a paper where all such equivalences are proved; but following the proofs for the scalar-valued case, one can see that everything works for the case of Hilbert-space valued case as well.}%
, so one can use the result from \cite{RdFra-OpValuedCZO-1986} without worrying about what definition of $H^1$ is used. 

In this paper we are considering the case when $\cX=\C$ and $\cY=\cH= \ell^2\otimes L^2(\Om, \bP)$, so we can say   that $K$ takes values in the Hilbert space $\cH$. 

Operators with such kernels also  act naturally  from $L^2(\R^N; \cK)\to L^2(\R^N; \cH\otimes\cK)$, where $\cK$ is a Hilbert space, and we will need this interpretation later in Section \ref{s2}. Indeed, with each vector $h\in \cH$ we can associate an operator $\cK \ni f \mapsto f\otimes h \in \cK\otimes \cH$ (and the norm of this operator is $\|h\|$).  
So, if  $K$ is ah $\cH$-valued  \cz kernel, then a \cz operator on $L^2(\R^N;\cK)$ is a bounded operator $T:L^2(\R^N;\cK)\to L^2(\R^N; \cH\otimes \cK)$ such that
$$
\La Tf, g\Ra_{L^2(\R^N; \cH\otimes \cK)} =\iint_{\R^N\times\R^N} \left\La K(x, y)\otimes f(y), g(y)\right\Ra_{\cH\otimes \cK} dy dx . 
$$
for all compactly supported $f$ and $g$ with separated supports. 

\subsubsection{Why $\bS$ is a \cz operator?}
To find the kernel $K$ of $\bS$ we need to compute $\bS \delta_y$, where $\delta_y$ is the unit mass at $y\in \R^N$:
$$
K(x, y)  = \left\{ \Delta_k^{\cD(\om)} \delta_y (x) \right\}_{k\in \Z,\ \om\in \Om} \in \ell^2\otimes L^2(\Om, \bP) 
$$
To get that expression rigorously, one needs to approximate $\delta_y$ by appropriate bump functions; notice that $\Delta_k^{\cD_\om} \delta_y$ is well defined, i.e.~does not depend on the choice of approximating sequence. 

Notice that for a dyadic cube $Q$, $\Delta_Q \delta_y (x) \ne 0$ only if $x\in Q, y\in Q$. Therefore $\Delta_Q \delta_y (x) = 0$ if $\ell(Q) < |x-y|_\infty$, so 
$\Delta_k^{\cD(\om)} \delta_y (x) = 0$ if $2^k < |x-y|_\infty$. 

Noticing that $\| \Delta_k^{\cD_\om} \delta_y \|_\infty \le C 2^{-kN}$ and summing the geometric series we get 
the property \cond1 of \cz kernels. 

To show property \cond2  notice that  $\Delta_k^{\cD(\om)} \delta_y(x) =\Delta_k^{\cD(\om)} \delta_y(x_0)$ if all 3 points $x, x_0, y$ are in the same cube $Q\in \cD(\om)$, $\ell(Q)=2^k$  and the points $x$, $x_0$ are in the same ``child'' of $Q$. 

The probability that it fails for a given $k$ can be estimated above by $C |x-x_0| 2^{-k}$. Using the estimate $\| \Delta_k^{\cD_\om} \delta_y \|_\infty \le C 2^{-kN}$ we conclude that 
$$
\int_\Om |\Delta_k^{\cD(\om)} \delta_y(x) - \Delta_k^{\cD(\om)} \delta_y(x_0)|^2 d\bP(\om) \le
C 2^{-2kN} |x-x_0| 2^{-k} =C|x-x_0| 2^{-2kN-k}.
$$

Let $k_0$ be the maximal $k\in \Z$ such that $2^k <\min\{|x-y|_\infty, |x_0-y|_\infty\}$. Then $\Delta_k^{\cD_\om} \delta_y(x) =\Delta_k^{\cD_\om} \delta_y(x_0) =0$ for $k<k_0$, so 
\begin{align*}
\lnm K(x, y) - K(x_0, y) \rnm^2 \le C \sum_{k\ge k_0} 2^{-2kN-k} |x-x_0|  & \le C |x-x_0| 2^{-2k_0 N -k_0} 
\\
& \le C |x-x_0| |x_0-y|^{-2N -1}. 
\end{align*}
Interchanging $x$ and $y$ and repeating the above reasoning we also get that 
$$
\lnm K(y,x) - K(y, x_0) \rnm^2  \le C |x-x_0| |x_0-y|^{-2N -1 }. 
$$
This means condition \cond2 holds with $\delta=1/2$. 

The proof for $\wt\bS$ is absolutely the same. 

\subsection{A remark about conditions $\bS\1 =0$, $\bS^*\1 =0$} Material in this section is not needed for the proof of the main results. However, it might be of interest for specialists; one can use it to present a different proof of the main results, without employing the cited above in Section \ref{s1.2.1} result from \cite{RdFra-OpValuedCZO-1986} result about \cz operators with operator-valued kernels. 

Note that the operator $\bS$, introduced above satisfies the conditions $\bS \1=0$ and $\bS^* 1=0$ (more precisely, the second condition should read as $\bS^* \1 e =0$ for all $e\in \cH$), which are well known to everybody familiar with $T(1)$-theorem.  

If one formally plugs $1$ into $\bS$ or $\bS^*$, the result will be $0$. Of course, it is only a formal reasoning, for $1$ is not in the domain of $\bS$, but any reasonable interpretation of $\bS\1$ gives the same result. For example it is not hard to show that 
\begin{equation}
\label{1.2}
\bS \1\ci Q \to 0, \quad \bS^*\1\ci Q e \to 0 \qquad \text{as } \ell(Q) \to \infty
\end{equation}
uniformly on compact subsets, where cubes $Q$ are centered at $0$. 

It is also easy to see that $\wt\bS \1=0$, but unfortunately  $\wt\bS^*\1\ne 0$. However, it is easy to modify $\wt\bS$ to make $\wt\bS^*\1=0$. 

Namely, let $\f$ be a function on the cube $[0,1)^N$ taking values $\pm1$ and such that $\int_Q\f dx =0$ and let $\f\ci Q(x) = \f((x-x\ci Q)/\ell(Q))$, where $x\ci Q$ is the base of $Q$, i.e. the point in $Q$ with smallest coordinates. 

Define the square function $\wt \bS$ by 
$$
\wt\bS f (k, \om, x) := \f\ci Q(x) U_Q \Delta_Q f \in \C^{2^N}, 
$$
where $Q\in \cD(\om)$ is the cube of size $2^k$ containing $x$. 

The function $\f\ci Q$ in the definition of $\wt\bS$ is introduced to insure that $\wt\bS^*1 =0$. Now it is easy to show that $\wt\bS\1=0$, $\wt\bS^*\1=0$ (in the sense of \eqref{1.2}). 

\cz operators satisfying $T\1=0$ and $T^*\1=0$ map $H^1\to H^1$. To show that one, for example can consider matrix of such an operator in the wavelet basis. It was shown in \cite{Meyer-WavOper-1992} that under rather mild assumption about wavelet basis, the coefficient space of $H^1$ in this basis is the Triebel--Lizorkin space $\fb^{0,2}_1$; see Appendix (Section \ref{s3}) for the definition. In \cite{Meyer-WavOper-1992} the scalar-valued case was treated, but one can easily see that everything works for the Hilbert-space valued case, and one just get the vector-valued space $\fb^{0,2}_1$.  Moreover, in \cite{Hyto-wavelets_2006} the $H^1$ spaces with values in UMD Banach spaces were characterized in terms of coefficient in the wavelet basis; in the Hilbert-space-valued case the result gives exactly $\fb^{0,2}_1$ with values in the Hibert space.

Using the standard estimates with \cz kernels one can see that if a \cz operator $T$ (even with the operator-valued kernel) satisfies $T\1=0$, $T^*\1=0$, then its matrix in the wavelet basis is what is called in \cite{Frazier-Jawerth-1990} \emph{almost diagonal} for $\fb^{0,q}_p$, $1\le p, q<\infty$; cf. Section \ref{s3} below for the definition. 

And it was shown in \cite{Frazier-Jawerth-1990} that \emph{almost diagonal} operators are bounded on all $\fb^{0,q}_p$, $1\le p, q<\infty$. Since the almost diagonality is a condition on the magnitude of the entries, the result holds for vector-valued Triebel--Lizorkin spaces.  
Of course, instead of considering a wavelet basis, one can consider a frame decomposition, given by what is called in \cite{Frazier-Jawerth-1990} \emph{$\f$-transform}; all the estimates will be the same. 

\section{Proof for the multi-parameter case}
\label{s2}
Proof of the main result for multi-parameter case (Theorem \ref{t0.4}) follows the  lines of the proof for the one parameter. 

Without loss of generality we can assume that the probability space $\Om$ is represented as a product, $(\Om, \bP)= (\Om_1\times\Om_2\times\ldots\times \Om_n, \bP_1\times\bP_2\times\ldots \times\bP_n)$ and that the random dyadic grid $\cD_k(\om)$ on $X_k$ depends only on $\om_k$. 

For a dyadic lattice $\cD = \cD_1\times \cD_2\times\ldots \times \cD_n$ on $X_1\times X_2\times\ldots\times X_n$ define the vector-valued square function $\bS_\cD$ taking values in $\ell^2(\Z^n)$
$$
\bS_\cD f (\bk, x) := \Delta_\bk^\cD f (x), \qquad \bk\in \Z^n,\ x\in \R^N
$$

Consider the space $\cL= L^1(\ell^2(\Z^n)\otimes L^2(\Om, \bP))$, 
$$
\| f\|_{\cL} := \int_{\R^N} \left( \int_\Om \sum_{\bk\in\Z^n} | f(\bk,\om, x)|^2 d\bP(\om) \right)^{1/2} dx. 
$$
It is an $L^1$ space with values in the Hilbert space $\ell^2(\Z^n)\otimes L^2(\Om, \bP)$. Note, that this Hilbert space can be decomposed as $\ell^2(\Z^n)\otimes L^2(\Om, \bP) =(\ell^2\otimes L^2(\Om_1 ,\bP_1))\otimes 
(\ell^2\otimes L^2(\Om_2,\bP_2))\otimes \ldots\otimes (\ell^2\otimes L^2(\Om_n,\bP_n))$; here $\ell^2=\ell^2(\Z)$. 

Define the vector-valued square function, taking values in the space $\ell^2(\Z^n)\otimes L^2(\Om, \bP)$ by 
$$
\bS f(\bk, \om, x) = \bS_{\cD(\om)} f (\bk, x), \qquad x\in\R^N, \ \bk\in \Z^n, \om\in \Om. 
$$
Note, that $\bS$ can be decomposed as a tensor product $\bS= \bS_1 \otimes \bS_2 \otimes\ldots \otimes \bS_n$, where $\bS_k$ is the one parameter square function defined by \eqref{1.1} in variables $x_k \in X_k =\R^{N_k}$, $\om_k\in \Om_k$. 

Clearly, as in the one parameter case, we have for the multi-parameter square functions $\|\bS f\|_1 =\int_X \left[ \E_\om (|S_{\cD(\om)} f(x)|^2) \right]^{1/2}dx$. 

Similarly, for the square function $\wt S$ one can construct its vector version  $\wt\bS$ with values in $\cH =\cH_1\otimes \cH_2\otimes \ldots \otimes \cH_n$, $\cH_k = \ell^2 \otimes L^2 (\Om_k, \bP_k) \otimes \C^{2^{N_k}}$. Again, it is easy to see that $\|\wt\bS f\|_{L^1(X;\cH)} = \int_X \left[ \E_\om (|\wt S_{\cD(\om)} f(x)|^2) \right]^{1/2} dx$.

As it was already discussed in Section \ref{s1.1}, operators $\bS_k$ are (Hilbert-space-valued) one-parameter \cz operators, so $\bS$ is the tensor product of such operators. And it is probably immediately clear to experts, that such operators map $H^1(X_1\otimes X_2\otimes \ldots\otimes X_n)$ to $L^1$.  

One way to see that is to notice that $\bS$ is a trivial case of multi-parameter \cz operators, and according to Theorem 2.2 in \cite{Pipher-Duke-1986} such operators map $H^1(X_1\otimes X_2\otimes \ldots\otimes X_n)$ to $L^1$.   Of course, one needs to use a Hilbert space valued version of the theorem, but it is clear to the specialists, that the proof from \cite{Pipher-Duke-1986} works in this case. It is also clear that while Theorem 2.2 in \cite{Pipher-Duke-1986} is stated for $\R\times \R\times \ldots \times \R$, the proof works for $X_1\otimes X_2\otimes \ldots\otimes X_n$, $X_k=\R^{N_k}$. 

For the reader who is not well familiar with multi-parameter $H^1$ spaces we present below an alternative proof, which exploits the tensor product structure of $\bS= \bS_1 \otimes \bS_2 \otimes\ldots \otimes \bS_n$; it uses only theory of  one parameter  $H^1$-spaces. Of course, we will need the theory of $H^1$-spaces with values in a Hilbert space, but we need the vector valued theory in the above multi-parameter reasoning as well. 

There is one more reason for the presenting the one-parameter proof below: the reasoning above gives the estimate $\int_X \left[ \E_\om (|S_{\cD(\om)} f(x)|^2) \right]^{1/2}dx \le C \| f \|_{H^1(X)}$. The opposite estimate follows from the inequality $\|f\|_{H^1} \le \int_X |S_\cD f (x)| dx $, which is trivial in one-parameter case.  In multi-parameter case, the same estimate, while true and known to specialists, requires some some work to prove it. 
The one-parameter  approach presented below gives a reasonably simple proof of this estimate.

\subsection{Multi-parameter $H^1$-spaces}
\label{s2.1}
Recall, cf \cite[III.4.4]{Stein-book_1993} that for  for one-parameter Hardy space $H^1(\R^N)$ the norm $\|S\ut L f\|_1$, where $S\ut L$ is the Lusin square function (aka Lusin Area Integral)
$$
S\ut L f (x) = \int_{\Gamma_x} |\nabla f (y, t)|^2 t^{1-n} dydt
$$
(here $\Gamma_{x}  := \{(y, t): y\in\R^N, t\ge 0, |y-x|< t\}$,   and $f(y, t)$ is the harmonic extension of $f$ from $\R^N$ to $\R^{N+1}_+ := \R^N\times \R_+ = \{(x, t): x\in \R^N, t\in \R_+\}$) gives an equivalent norm.  

One can consider the \emph{vectorization} $\bS\ut L $ of $S\ut L$ as follows. Let $\Gamma = \Gamma_0$ and define 
$$
\bS\ut L f (x, y, t) = t^{(1-n)/2} \nabla f (x+y, t), \qquad x\in \R^N, (y, t)\in \Gamma. 
$$
By construction $\bS\ut L f (x, \fdot, \fdot)\in L^2(\Gamma)\otimes \C^2$ and $\|\bS\ut L f (x, \fdot, \fdot)\|_{L^2(\Gamma)\otimes \C^2} = S\ut L f(x)$, therefore $\|\bS\ut L f\|_{L^1(L^2(\Gamma))} = \|S\ut L f\|_1$.  

One can define multi-parameter square functions $\vec S\ut L$ and $\vec\bS\ut L$ by 
\begin{align*}
\vec S\ut L  f (x_1, x_2, \ldots, x_n) & := 
\\
&  \left[ \int_{\Gamma_{x_1}\times\Gamma_{x_2}\times\ldots\times \Gamma_{x_n}} |\nabla_1\nabla_2\ldots\nabla_n f (y_1, t_1,  y_2, t_2, \ldots, y_n , t_n)|^2 \times \right.
\\&\qquad \qquad \qquad \left. 
\vphantom{\int_{\Gamma_{x_1}\times\Gamma_{x_2}\times\ldots\times \Gamma_{x_n}}}
\times t_1^{1-N_1} t_2^{(1-N_2)} \ldots t_n^{1-N_n} dy_1dt_1 dy_2 dt_2 \ldots dy_n dt_n \right]^{1/2}; 
\end{align*}
here $f (y_1, t_1,  y_2, t_2, \ldots, y_n , t_n)$ is the harmonic in each variable $(y_k, t_k)$, $y_k\in \R^{N_k}$, $t_k\in \R_+$ extension of $f$ from $\R^{N_1}\times \R^{N_2} \times\ldots\times \R^{N_n}$ to $\R^{N_1+1}_+\times \R^{N_2+1}_+ \times\ldots\times \R^{N_n+1}_+$ and $\nabla_k $ is the gradient in the variable $(y_k, t_k)$.

Following \cite{Chang-Fef-prod-1985} we say that $f\in H^1(X)= H^1(X_1\otimes X_2\otimes \ldots \otimes X_n)$ if $\vec S\ut L f \in L^1(X)$ and $\|\vec S\ut L f\|_1$ defines one of the possible equivalent norms in $H^1(X)$.

We also define the vector-valued linear square function $\vec \bS\ut L $ as $\vec\bS f (x) = \bS_1\ut L \otimes \bS_2\ut L \otimes \ldots \otimes \bS_n\ut L f (x) \in \cH = \cH_1\otimes \cH_2\otimes\ldots \otimes \cH_n$, where $\bS_k\ut L$ it one-parameter square function defined above taken in the variable $x_k$, $\cH_k= L^2(\Gamma_k)\otimes \C^2 $ and $\Gamma_k$ is the cone in $\R^{N_k+1}_+$ with the vertex at $0$.  

Again, by the construction $ \| \vec\bS\ut L f (x)\|_{\cH} = |\vec S\ut L f(x)|$, so $\|\vec\bS\ut L f \|_{L^1(X;\cH)}$ gives the norm in $H^1(X)$. 

We will use $\vec\bS\ut L $ to define multi-parameter $H^1$ with values in a Hilbert space $\cK$; in this case $\vec\bS\ut L f(x) \in \cH\times \cK$. We introduced such spaces only for notational purposes, so while most of the theory of multi-parameter $H^1$-spaces can be transfered to the Hilbert-space-valued case, we do not need this. 

\subsection{Proof of estimate $\int_X \left[ \E_\om (|S_{\cD(\om)} f(x)|^2) \right]^{1/2}dx \le C \|f\|_{H^1(X)}$} 
\label{s2.2}
Consider a multi-param\-e\-ter square function $\wt\bS_2 \otimes \ldots \otimes \wt\bS_n$, where each $\wt \bS_k$, $2\le k\le n$ is either one-parameter $\bS\ut L$ or one-param\-e\-ter ``random'' square function $\bS$,   defined in \eqref{1.1}, taken in the variable $x_k$. 

Assume that the choice of $\wt\bS_k$ is fixed. For a scalar-valued $f$ the function $\wt\bS_2 \otimes \ldots \otimes \wt\bS_n f$ takes values in $\wt\cH^1:=\wt\cH_2\otimes\ldots \otimes \wt\cH_n$, where each $\wt\cH_k$ is either $L^2(\Gamma_k)\otimes\C^2$ or $\ell^2\otimes L^2(\Om_k, \bP_k)$, depending on what square function $\wt\bS_k$ is. 

Let $\bS_k\ut L$, $\bS_k$ be the Lusin and ``random'' square functions, taken in the variable $x_k$, and let $\cH_k:= L^2(\Gamma_k)\otimes \C^2$ and $\cH_k':= \ell^2\otimes L^2(\Om_k, \bP_k)$ be the corresponding target spaces. 

\begin{lm}
\label{l2.1}
$$
\int_X \| \bS_1\otimes \wt\bS_2 \otimes \ldots \otimes \wt\bS_n f(x)\|\ci{\cH_1' \otimes \wt\cH^1} dx 
\le C 
\int_X \| \bS_1\ut L \otimes \wt\bS_2 \otimes \ldots \otimes \wt\bS_n f(x)\|\ci{\cH_1 \otimes \wt\cH^1} dx
$$
\end{lm}

Since  the tensor products of square functions we consider does not depend on the order (the square functions, taken in different variables obviously commute), the above lemma tells us that one can replace a factor  $\bS_k\ut L$ by  $\bS_k$ in $\wt\bS_1 \otimes \wt\bS_2 \otimes \ldots \otimes \wt\bS_n f$ and increase the norm by at most the factor $C$. 

Starting with $\bS_1\ut L \otimes \bS_2\ut L \otimes \ldots \otimes \bS_n\ut L  f$ and applying  Lemma \ref{l2.1} successively to each factor, we get 
$$
\int_X \| \bS_1\otimes \bS_2 \otimes \ldots \otimes \bS_n f(x)\|\ci{\cH'} dx 
\le C 
\int_X \| \bS_1\ut L \otimes \bS_2\ut L  \otimes \ldots \otimes \bS_n\ut L  f(x)\|\ci{\cH} dx, 
$$
which is exactly the desired estimate (here $\cH' = \cH_1'\otimes \cH_2' \otimes \ldots\otimes \cH_n'$ and $\cH = \cH_1\otimes \cH_2 \otimes \ldots\otimes \cH_n$).

\begin{proof}[Proof of Lemma 
\ref{l2.1}]
Let us introduce  notation $x=(x_1, x^1)\in X=X_1\times X^1$, where $x^1 = (x_2, x_3, \ldots, x_n) \in X^1 := X_2\times\ldots\times X_n$. 

Consider the vector-valued function
$$
\bS_1\ut L \otimes \wt \bS_2 \otimes \ldots \otimes \wt\bS_n  f = \bS_1\ut L \otimes( \wt\bS_2 \otimes \ldots \otimes \wt\bS_n ) f. 
$$ 
If
$$ 
\int_X \| \bS_1\ut L \otimes( \wt\bS_2 \otimes \ldots \otimes \wt\bS_n ) f(x)\|\ci{\cH_1'\otimes \wt\cH^1} dx
<\infty, 
$$
we conclude that for almost all $x^1$
\begin{equation}
\label{2.1}
\wt\bS_2 \otimes \ldots \otimes \wt\bS_n  f(\fdot, x^1) \in H^1(X_1; \wt\cH^1), 
\end{equation}
and
\begin{align}
\label{2.2}
\int_X \| \bS_1\ut L \otimes( \wt\bS_2 \otimes \ldots 
& 
\otimes \wt\bS_n ) f(x)\|\ci{\cH_1\otimes \wt\cH^1} dx 
\\ & \notag
=  
\int_{X^1} \| \wt\bS_2 \otimes \ldots \otimes \wt\bS_n  f(\fdot, x^1) \|\ci{ H^1(X_1;\wt\cH^1)} dx^1.
\end{align}
Note, that we have in \eqref{2.1} the usual one-parameter $H^1$-space (although vector-valued). 

As we discussed above in Section \ref{s1.1}, $\bS_1$ is a vector-valued \cz operator, so it maps one-parameter $H^1$ to $L^1$ (even in the Hilbert-space-valued case), so for almost all $x^1$ we have
\begin{align}
\label{2.3}
\int_{X_1} \| \bS_1 \otimes( \wt\bS_2 \otimes \ldots &\otimes \wt\bS_n ) f(x_1, x^1)\|\ci{\cH_1'\otimes \wt\cH^1} dx_1 
\\  \notag
& \le C
\|  \wt\bS_2 \otimes \ldots \otimes \wt\bS_n  f(\fdot, x^1) \|\ci{ H^1(X_1; \wt\cH^1) } 
\\  \notag
& =C \int_{X_1} \| \bS_1\ut L \otimes( \wt\bS_2 \otimes \ldots \otimes \wt\bS_n ) f(x_1, x^1)\|\ci{\cH_1\otimes \wt\cH^1} dx_1
\end{align}
Integrating over $X^1$ and taking into account \eqref{2.2}, we get the conclusion of the lemma. 
\end{proof}

\subsection{Estimate $\|f\|\ci{H^1(X)} \le C \int_X |S_\cD f(x)| dx$}
\label{s2.3} 
Proof of this estimate follows the lines of Section \ref{s2.2} almost word by word. It is based on the following analogue of Lemma \ref{l2.1}, which allows us replace one-parameter Lusin square functions by the dyadic ones. 

Consider again a multi-param\-e\-ter square function $\wt\bS_2 \otimes \ldots \otimes \wt\bS_n$, where now each $\wt \bS_k$, $2\le k\le n$ is either one-parameter $\bS\ut L$ or one-param\-e\-ter dyadic square function $\bS_{\cD_k}$,   defined in \eqref{1.1}, taken in the variable $x_k$. We assume here that in each $X_k$ dyadic lattices $\cD_k$ are fixed. 

We will use the same notation as in Section \ref{s2.2}, with the only exception that now $\cH_k'=\ell^2$

\begin{lm}
\label{l2.2}
$$
\int_X \| \bS_1\ut L \otimes \wt\bS_2 \otimes \ldots \otimes \wt\bS_n f(x)\|\ci{\cH_1 \otimes \wt\cH^1} dx
\le C 
\int_X \| \bS_{\cD_1}\otimes \wt\bS_2 \otimes \ldots \otimes \wt\bS_n f(x)\|\ci{\cH_1' \otimes \wt\cH^1} dx .
$$
\end{lm}

Applying  Lemma \ref{l2.1} successively to each variable, as we did in Section \ref{s2.2}, we get 
$$
\int_X \| \bS_1\ut L \otimes \bS_2\ut L  \otimes \ldots \otimes \bS_n\ut L  f(x)\|\ci{\cH} dx
\le C 
\int_X \| \bS_{\cD_1}\otimes \bS_{\cD_2} \otimes \ldots \otimes \bS_{\cD_n} f(x)\|\ci{\cH'} dx , 
$$
which is exactly what we need.

\begin{proof}[Proof of Lemma \ref{l2.2}]
Similarly to \eqref{2.2} we get 
\begin{align}
\label{2.4}
\int_X \| \bS_{\cD_1} \otimes( \wt\bS_2 \otimes \ldots 
& 
\otimes \wt\bS_n ) f(x)\|\ci{\cH_1\otimes \wt\cH^1} dx 
\\ & \notag
=  
\int_{X^1} \| \wt\bS_2 \otimes \ldots \otimes \wt\bS_n  f(\fdot, x^1) \|\ci{ H^1_{\cD_1} (X_1;\wt\cH^1)} dx^1.
\end{align}

We will now use the fact that for Hilbert-space-valued functions $\|\f\|_{H^1} \le C \|\f\|_{H^1_\cD}$. Again, as in the scalar-valued case, it follows from 
 $H^1$-BMO duality ($(H^1)^*=\text{BMO}$, $(H^1_\cD)^*=\text{BMO}_\cD$) and the trivial inclusion $\text{BMO}\subset\text{BMO}_\cD$, which  imply the inclusion $H^1_\cD \subset H^1$ with the desired estimates of the norms.  
 
Using this inequality we get that for almost all $x^1$
$$
\| \wt\bS_2 \otimes \ldots \otimes \wt\bS_n  f(\fdot, x^1) \|\ci{ H^1 (X_1;\wt\cH^1)}
\le C
\| \wt\bS_2 \otimes \ldots \otimes \wt\bS_n  f(\fdot, x^1) \|\ci{ H^1\ti d (X_1;\wt\cH^1)}.
$$
Integrating over $X^1$ and using \eqref{2.2}, \eqref{2.4} we get the conclusion of the lemma. 
\end{proof}

\section{Appendix: some facts about $H^1$ and BMO spaces.}
\label{s3}
\subsection{Hilbert-space-valued  BMO spaces}
\label{s3.1}
Let us recall that a function on $X=\R^N$ with values in a Hilbert space $\cH$ belongs to the space $\text{BMO}=\text{BMO}(X, \cH)$ if
\begin{equation}
\label{3.1}
\|f\|\ti{BMO} := \sup_{Q} \int_Q \|f(x)-f\ci Q\|\ci\cH dx <\infty;
\end{equation}
here $f\ci Q := |Q|^{-1} \int_Q f(x) dx$ and the supremum is taken over all cubes $Q\subset \R^N$. 

If we fix a dyadic lattice $\cD$ and take the supremum in\eqref{3.1} only over \emph{dyadic} cubes $Q\in\cD$, we get the dyadic  space $\text{BMO}_\cD$ associated with this lattice. 

It is well known that $(H^1(\R^N;\cH))^* = \text{BMO}(\R^N;\cH)$ and $(H^1_\cD(\R^N;\cH))^* = \text{BMO}_\cD(\R^N;\cH)$; any standard proof of $H^1$-BMO duality would work for the Hilbert-space-valued functions. 

\subsection{Triebel--Lizorkin spaces  $\fb^{\alpha, q}_p$ and equivalence of $\| S f\|_1$ and $\|\wt S f\|_1$}
In this section we fix a dyadic lattice $\cD$, for example take for $\cD$ the standard dyadic lattice. 
\subsubsection{Spaces $\fb^{\alpha, q}_p$}
Following \cite{Frazier-Jawerth-1990} define the  space $\fb^{\alpha, q}_p$, ($\alpha\in\R$, $1\le p, p<\infty$,  consisting of sequences $s=\{s\ci Q\}\ci{Q\in\cD}$ such that 
$$
\|s\|_{\fb^{\alpha, q}_p} := \left\| \left( \sum_{Q\in \cD} (|Q|^{-\alpha /n} |s\ci Q| \cdot |Q|^{-1/2}\1_Q )^q \right)^{1/q} \right\|_{L^p}<\infty
$$
For $p=\infty$ the norm is defined using BMO-like norm
$$
\|s\|_{\fb^{\alpha, q}_\infty} := \sup_{P\in\cD} \left( \frac{1}{|P|} \int_P \sum_{Q\in \cD, \, Q\subset P} 
(|Q|^{-\alpha /n} |s\ci Q| \cdot |Q|^{-1/2}\1_Q )^q \right)^{1/q}
$$

We are interested in the case when the smoothness parameter $\alpha=0$; to simplify the notation in this case  we will use $\fb^q_p:= \fb^{0, q}_p$. 

We will need the following facts about duality for spaces $\fb^{\alpha, q}_p$:
$$
(\fb^{\alpha, q}_p)^* = \fb^{-\alpha, q'}_{p'}, \qquad 1\le p , q<\infty;
$$
here $1/p+1/p'=1$, $1/q+1/q'=1$.

\subsubsection{Almost diagonal operators}
Following \cite{Frazier-Jawerth-1990} we say that an operator $A$ with matrix \linebreak
$\{a\ci{Q, P}\}\ci{Q, P\in \cD}$ is \emph{almost diagonal} (for spaces $\fb^q_p = \fb^{0, q}_p$) if there exists $\e>0$ and $C<\infty$, such that 
$$
|a\ci{Q,P}| \le C \left( 1+ \frac{|x\ci P - x\ci Q|}{\max\{ \ell(P), \ell(Q)\} } \right)^{-N-\e} \times \ \min
\left[ 
\left(\frac{\ell(Q)}{\ell(P)} \right)^{(N+\e)/2}, 
\left(\frac{\ell(P)}{\ell(Q)} \right)^{(N+\e)/2}   
\right]
$$
(the definition is a bit more complicated for $\fb^{s, q}_p$ with $s\ne 0$)

It was shown it \cite{Frazier-Jawerth-1990} that an almost diagonal operator is bounded in all $\fb^q_p$ spaces, $1\le, q, p <\infty$. 

\subsubsection{Equivalence of $\| S f\|_1$ and $\|\wt S f\|_1$}
From the above result one can easily obtain the equivalence of $\| S f\|_1$ and $\|\wt S f\|_1$. First, since $\max_Q |\Delta_Q f |^2 \le 2^N E_Q(|\Delta_Q f |^2)$ we have pointwise estimate $S f(x) \le 2^{N/2} \wt S f(x)$ and so $\| S f \|_1 \le 2^{N/2} \| \wt S f \|_1$. 

To get the estimate $\|\wt S f \|_1\le C\|Sf\|_1$ let us express the conditions $Sf\in L^1$, $\wt Sf\in L^1$ in terms of Triebel-Lizorkin space $\fb_1^2$. Namely, with each function $f\in L^1\ti{loc}$ let us associate 2 sequences $a=\{a\ci Q\}\ci{Q\in\cD}$ and $b=\{b\ci Q\}\ci{Q\in\cD}$
$$
a\ci Q = \left[ \left( E\ci Q |\Delta\ci Q f|^2 \right)(x) \right]^{1/2}, \qquad b\ci Q = \left(\Delta\ci R f \right)(x), 
$$
where $R$ is the ``parent'' of $Q$ and $x$ is an arbitrary point in $Q$ (the result does not depend on $x$). 
Then clearly
$$
\| \wt S f \|_1 = \| a \|_{\fb_1^2}, \qquad \|  S f \|_1 = \| b \|_{\fb_1^2}.
$$
Note that 
$$
a\ci R = \left( 2^{-N} \sum_{Q\text{ is child of } R} | b\ci Q|^2 \right)^{1/2} \le 2^{-N/2} \sum_{Q\text{ is child of } R} | b\ci Q| =: T|b|, 
$$
where $|b|:=\{ | b\ci Q | \}\ci{Q\in\cD}$ and for $s=\{s\ci Q\}\ci{Q\in\cD}$
$$
(Ts)\ci R := 2^{-N/2} \sum_{Q\text{ is child of } R} s\ci Q. 
$$
The operator $T$ is almost diagonal (it has only finitely many ``diagonals''), so
$$
\| \wt S f \|_1 = \| a \|_{\fb_1^2} \le \| T|b|\|_{\fb_1^2} \le C \| \,|b|\,\|_{\fb_1^2} = 
C \| b \|_{\fb_1^2} =  C\|  S f \|_1. 
$$
\ \hfill\qed
%\newpage

\def\cprime{$'$}
\providecommand{\bysame}{\leavevmode\hbox to3em{\hrulefill}\thinspace}

\end{document}